\documentclass[]{amsart}
\usepackage[utf8]{inputenc}
\usepackage{amsfonts}
\usepackage{hyperref}
\hypersetup{
pdftitle={Killing Fields on \MakeLowercase{\textit{m}}-Quasi-Einstein Manifolds of Constant Scalar Curvature},
pdfsubject={Differential Geometry, Mathematical Relativity},
pdfauthor={Eric Cochran},
pdfkeywords={Killing Fields, Differential Geometry, Mathematical Relativity, \textit{m}-Quasi-Einstein Manifolds, Near Horizon Geometries}
}
\usepackage{amsmath}
\usepackage{verbatim}
\usepackage{xcolor}
\usepackage{amsthm}
\usepackage{pdflscape}
\usepackage{pgfplots}
\usepackage{mathrsfs}

\makeatletter
\def\thmheadbrackets#1#2#3{%
  \thmname{#1}\thmnumber{\@ifnotempty{#1}{ }\@upn{#2}}%
  \thmnote{ {\the\thm@notefont[#3]}}}
\makeatother

\newtheoremstyle{brakets}
  {}
  {}
  {\itshape}
  {}
  {\bfseries}
  {.}
  { }
  {\thmheadbrackets{#1}{#2}{#3}}

\theoremstyle{brakets}
\newtheorem{thm}{Theorem}[section]
\newtheorem{cor}[thm]{Corollary}

\newtheorem{lem}[thm]{Lemma}

\theoremstyle{definition}

\newtheorem{rk}[thm]{Remark}

\newcommand{\Div}{\mathrm{div}}
\newcommand{\Ric}{\mathrm{Ric}}
\newcommand{\Scal}{R}
\newcommand{\Hess}{\mathrm{Hess}}

\theoremstyle{remark}

\pgfplotsset{compat=1.18}

\numberwithin{equation}{section}

\begin{document}

\title{Killing Fields on Compact \MakeLowercase{\textit{m}}-Quasi-Einstein Manifolds}

\author{Eric Cochran}

\begin{abstract}
We show that given a compact, connected $m$-quasi Einstein manifold $(M,g,X)$ without boundary, the potential vector field $X$ is Killing if and only if $(M, g)$ has constant scalar curvature.  This extends a result of Bahuaud-Gunasekaran-Kunduri-Woolgar, where it is shown that $X$ is Killing if $X$ is incompressible.  We also provide a sufficient condition for a compact, non-gradient $m$-quasi Einstein metric to admit a Killing field.  We do this by following a technique of Dunajski and Lucietti, who prove that a Killing field always exists in this case when $m=2$.  This condition provides an alternate proof of the aforementioned result of Bahuaud-Gunasekaran-Kunduri-Woolgar.  This alternate proof works in the $m = -2$ case as well, which was not covered in the original proof.
\\

\end{abstract}

\maketitle

\section{Introduction}
A triple $(M, g, X)$ is said to satisfy the $m$-quasi Einstein equation if
   \begin{equation} 
   \label{qe}
   \textrm{Ric} + \frac{1}{2}\mathcal{L}_X g  - \frac{1}{m} X^* \otimes X^* = \lambda g.
   \end{equation}

\noindent Here, we let $X$ be a smooth vector field on the Riemannian manifold $(M, g)$, $X^*$ be the 1-form dual to $X$, $\mathcal{L}_X g$ be the Lie derivative of the metric along $X$, and $m \neq 0$ be a real number.  We call a solution \textit{trivial} if $X=0$.  Quasi-Einstein metrics when $m=2$ are called \textit{near horizon geometries} and are of particular interest in general relativity.  When $m=\infty$, one can interpret (1.1) as: 

$$\textrm{Ric} + \frac{1}{2}\mathcal{L}_X g = \lambda g.$$  

\noindent Solutions to the quasi-Einstein equation in this case are called \textit{Ricci solitons}, which are important in the study of Ricci flow.  This real parameter $m$ can therefore be viewed as interpolating between these two important special cases. \\

A number of rigidity results concerning quasi-Einstein metrics have been obtained.  In [\ref{b6}], Case-Shu-Wei study \textit{gradient solutions} to (1.1), i.e. solutions where $X=\nabla\phi$ for some smooth function $\phi$ on $M$.  One of their results says that there do not exist any non-trivial, compact, gradient solutions to (1.1) with constant scalar curvature.  We will see that this result follows the main result of this paper as well.  More generally, in [\ref{b7}], Bahuaud-Gunasekaran-Kunduri-Woolgar study closed, compact solutions to (1.1).  That is, they study solutions with $dX^* = 0$.  For $\lambda > 0$, they show that $X$ must be gradient. For $\lambda = 0$, Chruściel-Reall-Tod in $[\ref{b9}]$ show that $X=0$ and $(M, g)$ is Ricci-flat.  When $\lambda < 0$, Bahuaud-Gunasekaran-Kunduri-Woolgar show that $\Div X = 0$ and $|X|$ is constant.  In the $\lambda < 0$ case, Wylie further shows in $[\ref{b4}]$ that all solutions to $(1.1)$ are either trivial, or they split as a product $(S^1 \times N, d\theta^2 + g_N)$ where $N$ is Einstein and $g_N$ is the metric on $N$.  If, in addition, constant scalar curvature is assumed, then these metrics of the form $(S^1 \times N, d\theta^2 + g_N)$ are the only compact $m$-quasi Einstein manifolds with $X \neq 0$ and $dX^* = 0$.  \\

In [\ref{b8}], Chen-Liang-Zhu study $m$-quasi Einstein metrics on Lie groups.  They show that given a compact Lie group $G$ with a left invariant metric satisfying (1.1), then $X$ is left invariant and Killing.  In [\ref{b3}], Lim generalizes this to the compact quotient of a Lie group by some discrete group of isometries $\Gamma$.  Using this, Lim also classifies compact, locally homogeneous solutions to (1.1) in dimension 3.  Some examples of homogeneous solutions to (1.1) in dimension 3 were previously found by Barros-Ribeiro-Filho in [\ref{b10}]. 
 Some solutions given by Lim have $\lambda > 0$.  Since homogeneous manifolds always have constant scalar curvature, this shows that examples do exist when $dX^* \neq 0$, unlike in the previous paragraph when $dX^* = 0$. \\

It is natural to ask under what conditions is the vector field $X$ in solutions to (1.1) Killing.  In [\ref{b1}], Bahuaud-Gunasekaran-Kunduri-Woolgar show that given a solution to (1.1) with $M$ closed, $\Div X = 0$, and $m \neq -2$, $X$ is Killing.  It follows that $(M,g)$ has constant scalar curvature in this case, which is noted in [\ref{b11}] when $m=2$ (their argument is easily generalized to $m \neq 2$).  In Corollary 1.4, we see that this result of Bahuaud-Gunasekaran-Kunduri-Woolgar extends to the case when $m= -2$. \\

In section 2, we complete the picture by proving the converse: If constant scalar curvature is assumed, then $X$ is Killing.  Putting this result together with the result of Bahuaud-Gunasekaran-Kunduri-Woolgar mentioned above, we obtain the following theorem, which is the main result in this paper:

\begin{thm}
Let $(M, g, X)$ be a solution to 
    \begin{equation} \mathrm{Ric} + \frac{1}{2}\mathcal{L}_X g  - \frac{1}{m} X^* \otimes X^* = \lambda g. \end{equation} such that $(M,g)$ is closed (i.e. compact and without boundary) and $m\neq -2$.  Then $(M, g)$ has constant scalar curvature if and only if $X$ is Killing.
\end{thm}

Observe that Theorem 1.1 can be viewed as a generalization of the result of Chen-Liang-Zhu and Lim mentioned above since locally homogeneous spaces clearly have constant scalar curvature.  \\

If $X=0$, then the $m$-quasi-Einstein condition reduces to the condition that $(M, g)$ is Einstein.  Also, Theorem 1.1 says that if $(M, g)$ is Einstein, $X$ is Killing, and in this case (1.1) reduces to

 \begin{equation} 
   \textrm{Ric} - \frac{1}{m} X^* \otimes X^* = \lambda g 
   \end{equation}    

If $(M, g)$ is Einstein in this case, then either $X=0$, or $X^* \otimes X^*$ is proportional to the metric.  If $n>1$, one can choose two distinct linearly independent vectors from an orthornormal basis $\{e_i\}$.  Plugging this into (1.3) yields
    \[ \frac{1}{m}g(X, e_i)g(X, e_j) = \lambda g(e_i,e_j) = 0 \]
for $i \neq j$.  This clearly implies $X = 0$.  Thus, $n=1$ in this case. 
 The fact that $X = 0$ for $n \geq 3$ in this case is shown in [\ref{b13}] by Barros-Gomes.  Since the only compact 1-dimensional Riemannian manifold is $\mathbb{S}^1$, we obtain the following corollary to Theorem 1.1, which is also proved in [\ref{b3}, Proposition 6.7].

\begin{cor}[Barros-Gomes, \ref{b13}, Lim, \ref{b3}]
Let $(M, g, X)$ be a solution to (1.1) with $M$ closed and connected and suppose $(M, g)$ is Einstein.  Then either $M=\mathbb{S}^1$ or $X = 0$.
\end{cor}

In section 3, we turn to a result due to Dunajski and Lucietti in [\ref{b2}], which states that given a non-gradient solution $(M, g, X)$ to (1.1), a non-trivial Killing field $K$ exists on $(M, g)$ when $m=2$.  Unfortunately, the technique used to show this does not seem to guarantee $K$ is Killing when $m \neq 2$.  Instead, we obtain a condition for when $K$ must be Killing, following the technique used in [\ref{b2}].  The condition is as follows:

\begin{thm}
Let $(M, g, X)$ be a solution to (1.1) with $M$ closed.  Let 
$$K = \frac{2}{m}\Gamma X + \nabla \Gamma$$ where $\Gamma > 0$ is a positive, smooth function defined on $M$ such that $\Div K = 0$.  Then $K$ is Killing if and only if

\[ (m-2) \int_{M} \Ric (\nabla \Gamma, K) \, dV = 0.\]
\end{thm}

In [\ref{b12}], Colling-Dunajski-Kunduri-Lucietti find new quasi-Einstein metrics on $\mathbb{S}^2$ for general $m$.  They also find a Killing field $K$ in the form of Theorem 1.3 with $\Gamma$ being a function of one of the coordinates.  This gives an explicit example of the integral condition in Theorem 1.3 being satisfied with $m \neq 2$ and $\nabla \Gamma \neq 0$. \\

The integral condition in Threorem 1.3 can be used to provide an alternate proof to the result of Bahuaud-Gunasekaran-Kunduri-Woolgar which also applies when $m=-2$.  This leaves us with the following corollary: 
\begin{cor}
Let $(M, g, X)$ be a solution to (1.1) with $M$ closed and suppose $\Div X = 0$.  Then $X$ is Killing. 
\end{cor}

\section{Constant Scalar Curvature}
    Before proving Theorem 1.2, we will start by proving a useful lemma. \\

     \begin{lem}
        Let $X$ be a smooth vector field on a compact Riemannian manifold $(M, g)$.  Then 
            \begin{equation}
               \int_M (\mathrm{\Div} X)^2 \, dV = -\int_M \nabla_X \mathrm{\Div} X \, dV,
            \end{equation} \label{int} where $dV$ is the volume form on $(M, g)$.
    \end{lem} 
\begin{proof}
    For a smooth vector field $X$ and for a choice of coordinates so that $\nabla_{e_i}e_i = 0$ at a specified point $p \in M$, one can compute
        \begin{align}  \notag \Div(\Div(X)X) &=  \nabla_{e_i}\Div(X)X^{i} \\  \label{id} &=  \Div(X)\nabla_{e_i}X^i + g(X, \nabla_{e_i}\Div X)  \\  \notag &= (\Div X)^2+\nabla_X \Div X. \end{align} 
    Observe that the left-most and right-most expressions of (2.2) are coordinate-free, and thus this holds at all points $p \in M$. 
    Integrating both sides of (2.2), one immediately obtains (2.1) since 
        \[ \int_{M}  \Div(\Div(X)X) \, dV = 0  \] by Stokes' Theorem. 
\end{proof}
    
     \textit{Proof of Theorem 1.1: }First, suppose $X$ is Killing and that $(M, g, X)$ is a compact solution to (1.1) and $m \neq -2$.  Taking the trace of (1.1) yields 
        \begin{equation}
            \Scal + \Div X-\frac{1}{m} |X|^2 = \lambda n.
        \end{equation}  Here, we use $R$ to denote scalar curvature.  This can be rearranged to 
        \[ |X|^2 = -\lambda nm + m \Scal + m \Div X.  \]  Since $X$ is Killing, $\Div X = 0$.  In [\ref{b1}],Bahuaud-Gunasekaran-Kunduri-Woolgar show that $|X|$ must have constant norm.  Thus, (2.3) immediately tells us that $\Scal = const.$ as desired. \\
        \indent Conversely, let $(M, g)$ have constant scalar curvature.  Letting $c := \lambda n - \Scal$, we can use (2.3) to write $|X|^2 = m \Div X - mc$.  By assumption, $(M, g)$ has constant scalar curvature, so $\Scal = const.$ and thus $c$ is also a constant.  By equation (2.24) of [\ref{b1}], one has the following identity for triples $(M, g, X)$ satisfying (1.1):
            \begin{multline}
                \frac{2}{m}\Delta(|X|)^2 + \frac{2}{m}\nabla_X (|X|)^2 - \left ( \frac{1}{m} + \frac{1}{2} \right ) |\mathcal{L}_X g|^2 = 
                \\ -\Delta \Div X+ \left ( \frac{4}{m} + 1  \right ) \nabla_X (\Div X) - 2 \lambda \Div X + \frac{4}{m^2}|X|^2 \Div X + \frac{2}{m}( \Div X)^2.
            \end{multline}
        
       \noindent Substituting $m \Div X - mc$ in for $|X|^2$ and simplifying, one gets
            \begin{multline}
                3\Delta \Div X +  \left ( 1 - \frac{4}{m} \right ) \nabla_X \Div X - \left ( \frac{1}{m} + \frac{1}{2} \right ) |\mathcal{L}_X g|^2 \\  =  \left ( -2\lambda - \frac{4c}{m} \right )\Div X + \frac{6}{m}(\Div X)^2.
            \end{multline}
        Next, we integrate over the closed manifold $M$ to obtain
            \begin{equation}
    \left ( 1 - \frac{4}{m} \right ) \int_M \nabla_X \Div X \, dV = \left ( \frac{1}{m} + \frac{1}{2} \right ) \int_M |\mathcal{L}_X g|^2 \, dV + \frac{6}{m} \int_M (\Div X)^2 \, dV.
\end{equation}
    Observe that several terms have vanished due to Stokes' Theorem.  One can now use Lemma 2.1 to rewrite (2.6) as
        \begin{equation}
             -\left ( 1 + \frac{2}{m} \right ) \int_M (\Div X)^2 \, dV =  \left (\frac{1}{2} + \frac{1}{m} \right ) \int_M |\mathcal{L}_X g|^2 \, dV.
        \end{equation}
    When $m \neq -2$, one can divide by the coefficient on the right hand side to obtain
            \begin{equation}
             -2 \int_M (\Div X)^2 \, dV =   \int_M |\mathcal{L}_X g|^2 \, dV.
        \end{equation}
    
    \noindent Since the two sides of (2.8) have opposite sign, they must be identically zero.  Hence, when $m \neq -2$, this implies that $|\mathcal{L}_X g| = 0$, i.e. $X$ is Killing as desired. \hfill \qedsymbol{} \\
    

\begin{rk}
    As mentioned in the introduction, Case-Shu-Wei show that there do not exist any compact, gradient solutions to (1.1) with constant scalar curvature.  We can view this result as a corollary of Theorem 1.1.  Indeed, if  $(M, g, X)$ is such a solution and $X=\nabla f$, $X$ is Killing by Theorem 1.1.  Hence, $\Hess f = 0$.  But that means that $f$ must be constant by the maximum principle, and so $X = \nabla f = 0$ and so the solution must be trivial as desired. 
\end{rk}

\begin{rk}
    If a compact m-Quasi Einstein manifold $(M, g, X)$ has constant scalar curvature, $X$ is Killing by Theorem 1.1, and $|X|$ is constant as shown in [\ref{b1}].  Since $X$ is Killing, we have that $\mathrm{Ric} - \frac{1}{m}X^* \otimes X^* = \lambda g$.  Observe that the eigenvalues of the Ricci tensor are $\lambda$ and $\lambda + \frac{|X|^2}{m}$, which are constant since $|X|$ is constant.  It is therefore reasonable to ask whether or not the Ricci tensor is parallel in this case.  We can show this is not necessarily the case. Taking a covariant derivative on both sides with respect to some vector field $Y$, we obtain $\nabla_{Y} \mathrm{Ric} = \frac{1}{m} \nabla_Y (X \otimes X) = \frac{1}{m}\nabla_{Y}X \otimes X + \frac{1}{m} X \otimes \nabla_{Y}X$.  Plugging in $X$ and  $\nabla_Y X$ into this tensor yields $(\frac{1}{m}\nabla_{Y}X \otimes X + \frac{1}{m} X \otimes \nabla_{Y}X) (X, \nabla_Y X) = \frac{2}{m}|X|^2|\nabla_Y X|^2$.  Thus, $\mathrm{Ric}$ is parallel if and only if $X$ is parallel.  Examples of solutions $(M, g, X)$ with $(M,g)$ having constant scalar curvature and $X$ not parallel can be found in [\ref{b3}].
\end{rk}

\section{Killing Fields in General}

In [\ref{b2}], Dunajski and Lucietti show that in the $m=2$ case, a non-gradient solution $(M, g, X)$ to (1.1) with $M$ closed admits a non-trivial Killing field.  In this section, we aim to generalize their argument for arbitrary $m \neq 0$, following the same technique.  Unfortunately, for the $m \neq 2$ case, this technique does not seem to necessarily yield a Killing field.  However, our generalization will yield a sufficient condition for a solution to $(M, g, X)$ to produce a non-trivial Killing field.  We begin by defining a new vector field $K$ as
    \begin{equation}
        K = \frac{2}{m} \Gamma X + \nabla \Gamma
    \end{equation}
\noindent where $\Gamma > 0$ is a positive smooth function such that $\Div K = 0$.  By [\ref{b2}, Lemma 2.2], such a function exists.  The choice of the coefficient of $\frac{2}{m}$ is chosen strategically so that certain cross terms cancel when using (3.1) to rewrite (1.1) in terms of $K$ and $\Gamma$.  To do this, we can solve for $X$ to obtain $$X = \frac{K - \nabla \Gamma}{\frac{2}{m}\Gamma}.$$  Plugging in this expression for $X$ into (1.1) and solving for the $\Ric$ term yields

    \begin{equation}
         \Ric = -\frac{m}{4}\mathcal{L}_{\frac{K - \nabla \Gamma}{\Gamma}} g + \frac{m}{4} 
         \left ( \frac{K^* - d\Gamma}{\Gamma} \right ) \otimes 
         \left ( \frac{K^* - d\Gamma}{\Gamma} \right )
         + \lambda g.
    \end{equation}

\noindent Note that the factor of $\frac{2}{m}$ has already been moved outside each term for convenience.  Using multilinearity, we can expand the term involving $X^* \otimes X^*$, and rewrite it as 

\begin{equation} \frac{m}{4} 
         \left ( \frac{K^* - d\Gamma}{\Gamma} \right ) \otimes 
         \left ( \frac{K^* - d\Gamma}{\Gamma} \right ) = \frac{m}{4\Gamma^2}(K^* \otimes K^* - K^* \otimes d\Gamma - d\Gamma \otimes K^* + d\Gamma \otimes d\Gamma).
    \end{equation}
To deal with the Lie derivative term, we will first recall that the Lie derivative can be viewed as the symmetrization of the covariant derivative, i.e. $\mathcal{L}_{X} g (Y, Z) = g(\nabla_{Y}X, Z) + g(\nabla_{Z}X, Y)$.  Writing $X$ in terms of $K$ and simplifying, we see that

\begin{align}
\notag \mathcal{L}_{X} g (Y, Z) &= 
\frac{m}{2} \mathcal{L}_{\frac{K - \nabla \Gamma}{\Gamma}} g (Y, Z) 
\\ \notag  &= \frac{m}{2\Gamma} (g(\nabla_{Y}  (K - \nabla \Gamma), Z) + g(\nabla_{Z} (K - \nabla \Gamma), Y))  \\ & \quad - \frac{m \nabla_{Y} \Gamma}{2\Gamma^2}g(K - \nabla \Gamma, Z) -\frac{m\nabla_{Z}\Gamma}{2\Gamma ^2} g(K - \nabla \Gamma, Y) \\ \notag &= 
\frac{m}{2\Gamma} (\mathcal{L}_{K} g (Y, Z) - 2\Hess \Gamma (Y, Z)) - \frac{m \nabla_{Y} \Gamma}{2\Gamma^2}g(K, Z)
\\ \notag & \quad - \frac{m \nabla_{Z} \Gamma}{2\Gamma^2}g(K, Y) + \frac{m (\nabla_{Y}\Gamma)(\nabla_{Z}\Gamma)}{2\Gamma^2}+\frac{m (\nabla_{Z}\Gamma)(\nabla_{Y}\Gamma)}{2\Gamma^2}
\\ \notag  &=
 \frac{m}{2\Gamma} (\mathcal{L}_{K} g (Y, Z) - 2\Hess \Gamma (Y, Z)) 
 \\ \notag & \quad- \frac{m}{2\Gamma^2} (d\Gamma \otimes K^* (Y, Z) +  K^* \otimes d\Gamma (Y, Z) - 2d\Gamma \otimes d\Gamma (Y, Z)). 
\end{align}

\noindent Multiplying (3.4) by $-1/2$, adding the result to (3.3), and canceling terms, we obtain
    \begin{equation}
         \Ric = \frac{m}{4\Gamma^2}K^* \otimes K^* - \frac{m}{4\Gamma^2} d \Gamma \otimes d \Gamma - \frac{m}{4\Gamma}\mathcal{L}_K g + \frac{m}{2 \Gamma} \Hess \Gamma + \lambda g
    \end{equation}
which can be rewritten in terms of the Lie derivative term as
    \begin{equation}
          \mathcal{L}_K g = \frac{1}{\Gamma}K^* \otimes K^* - \frac{1}{\Gamma}d \Gamma \otimes d\Gamma - \frac{4\Gamma}{m}\Ric + 
            2\Hess \Gamma + \frac{4\Gamma \lambda}{m}g.
    \end{equation} \label{lie}

\noindent We now take the divergence of both sides of (3.6) to obtain

\begin{align}
         \notag \Div (\mathcal{L}_K g) &= \Div \left ( \frac{1}{\Gamma}K^* \otimes K^* \right )  - \Div \left ( \frac{1}{\Gamma}d \Gamma \otimes d\Gamma \right ) - \Div \left ( \frac{4\Gamma}{m}\Ric \right ) \\& \quad + 
         \Div \left (  2\Hess \Gamma \right ) + \Div \left ( \frac{4\Gamma \lambda}{m}g \right ).
\end{align}

To compute these individual divergences, we will repeatedly use the identity $\Div (\phi K) = \phi \Div K + g(\nabla \phi, K)$ for any smooth function $\phi$.  However, since $\Div K = 0$ by assumption, we have that $\Div (\phi K) = g(\nabla \phi, K)$.  We therefore compute

\begin{align}
    \Div \left ( \frac{1}{\Gamma}K^* \otimes K^* \right )  &=  \frac{1}{\Gamma}(\nabla_{K}K)^* - \frac{1}{\Gamma^2}g(\nabla \Gamma, K)K^*\\
    \Div \left (- \frac{1}{\Gamma} d\Gamma \otimes d\Gamma \right )  &=  -\frac{1}{\Gamma}\Delta \Gamma d\Gamma - \frac{1}{\Gamma}(\nabla_{\nabla \Gamma}\nabla \Gamma)^* + \frac{|\nabla \Gamma|^2}{\Gamma^2} d\Gamma\\ 
    \frac{-4}{m}\Div \left ( \Gamma \Ric \right ) &= -\frac{2}{m}\Gamma d\Scal - \frac{4}{m}\Ric(\nabla \Gamma)^* \\
    2 \Div (\Hess \Gamma) &= 2 \Ric(\nabla \Gamma)^* + 2(\nabla \Delta \Gamma)^*,
\end{align}

\noindent where in (3.10), we apply the twice contracted Ricci identity.  Now, we will contract equations (3.8)-(3.11) with $K$, add them, and throw back in the last term we neglected to obtain

\begin{align} \Div (\mathcal{L}_K g) (K) &= \frac{1}{\Gamma}g(\nabla_{K}K, K) - \frac{|K|^2}{\Gamma^2}g(\nabla \Gamma, K) - \frac{1}{\Gamma}\Delta \Gamma g(\nabla \Gamma, K) - \frac{1}{\Gamma}g(\nabla_{\nabla \Gamma} \nabla \Gamma, K) 
    \\ \notag & \quad + \frac{|\nabla \Gamma|^2}{\Gamma^2}g(\nabla \Gamma, K) - \frac{2}{m}\Gamma g(\nabla R, K) -\frac{4}{m}g(\Ric(\nabla \Gamma), K) \\ \notag & \quad + 2g(\Ric(\nabla \Gamma), K) 
      + 2g(\nabla \Delta \Gamma, K) +\frac{\lambda}{m} \Div (\Gamma g)(K). 
\end{align}

The last step is to integrate both sides of (3.12).  The condition that $K$ is Killing is equivalent to this integral being zero.  To see why, we have the identity
    \begin{equation}
        \int_{M} \Div(\mathcal{L}_K g) (K) \, dV = -\frac{1}{2}\int_{M} |\mathcal{L}_K g|^2 \, dV.
    \end{equation}
    
\noindent For a proof, see [\ref{b5}, Lemma 3.5].  Setting (3.13) equal to zero is clearly equivalent to requiring that $|\mathcal{L}_{K} g|^2 = 0$ (i.e. $K$ is Killing).  Integrating both sides of (3.12), we note that the integrand on the right hand side reduces drastically.  Firstly, the integral of the right-most term on the right-hand side of (3.12) vanishes since  

    \begin{equation}
        \int_{M} \Div \left ( \frac{4 \Gamma \lambda}{m} g \right ) (K) \, dV = \frac{4 \lambda}{m}\int_M (\Gamma \Div (g) (K) + g(\nabla \Gamma, K)) \, dV = 0.
    \end{equation}

Furthermore, we have that 
    \begin{equation}
        2g(\nabla \Delta \Gamma, K) = 2\Div (\Delta \Gamma K)
    \end{equation}

\noindent since we are assuming $\Div K = 0$.  Thus, this term integrates to zero by Stokes' Theorem.  Next, we will show that many of these terms will cancel with the term $\frac{4}{m}g(\Ric(\nabla \Gamma), K)$.  To see why and how, we will first consider the trace of equation (3.6).  This yields

    \begin{equation}
        2\Div K = \frac{1}{\Gamma}|K|^2-\frac{1}{\Gamma}|\nabla \Gamma|^2-\frac{4 \Gamma}{m}\Scal + 2\Delta \Gamma + \frac{4\Gamma \lambda n}{m}.
    \end{equation}

\noindent Dividing the above by $\Gamma$ (which is possible since $\Gamma > 0$ by assumption), recalling $\Div K = 0$, and rearranging, we obtain

    \begin{equation}
        \frac{4}{m}R = \frac{1}{\Gamma^2}|K|^2-\frac{1}{\Gamma^2}|\nabla \Gamma|^2 + 2\frac{\Delta \Gamma}{\Gamma} + \frac{4\lambda n}{m}.
    \end{equation}

\noindent Taking the covariant derivative of both sides, multiplying by $\frac{\Gamma}{2}$, and contracting with $K$, one obtains

    \begin{align} \frac{2}{m}\Gamma g(\nabla R, K) 
      \notag & =  \frac{1}{\Gamma}g(\nabla_{K}K, K) - \frac{1}{\Gamma}g(\nabla_{\nabla \Gamma}\nabla \Gamma, K)  \\  \notag & \quad - 
     \frac{|K|^2}{\Gamma^2}g(\nabla \Gamma, K) + 
            \frac{|\nabla \Gamma|^2}{\Gamma^2}g(\nabla \Gamma, K)   \\ & \quad + 
       g(\nabla \Delta \Gamma, K) - \frac{\Delta \Gamma}{\Gamma}g(\nabla \Gamma, K).
    \end{align}

\noindent Substituting the right hand side of equation (3.18) into equation (3.12), one observes that almost all the terms will either cancel or integrate to zero (by Stokes' Theorem), with the exception of the two terms involving $\Ric (\nabla \Gamma, K)$.  After cancelling terms in (3.12), integrating on both sides, and recalling (3.13), one obtains

    \begin{equation}
         -\frac{1}{2} \int_{M} |\mathcal{L}_{K} g|^2 \, dV =  \int_{M} \Div( \mathcal{L}_{K} g) (K) = \frac{2m-4}{m} \int_{M} \Ric(\nabla \Gamma, K) \, dV.
    \end{equation}

\noindent Hence, $K$ is Killing exactly when 

    \begin{equation}
        \frac{2m-4}{m} \int_{M} \Ric(\nabla \Gamma, K) \, dV = 0.
    \end{equation}

\noindent Observe that $K$ is always Killing when $m=2$, which was shown in [\ref{b2}] by Dunajski and Lucietti.  Multiplying (3.20) by $m/2$ proves Theorem 1.3. \\

 \textit{Proof of Corollary 1.4: }As an interesting special case, consider the case where $\Div X = 0$.  In this case, we can take $\Gamma = const.$, and $K$ is some constant multiple of $X$.  Therefore $\Ric (\nabla \Gamma, K) = \Ric (0, K) = 0$.  Hence, the integral (3.20) is zero for any $m \neq 0$.  This shows that $X$ is Killing whenever $\Div X = 0$.  This is an alternate method of proving a result which was obtained in [\ref{b1}]. \\

\noindent \textbf{Acknowledgements. }The author of this paper would like to thank his advisor, Will Wylie, for his continuous support and helpful feedback he has given.  He would also like the thank James Lucietti, Eric Bahuaud, and the referee for taking the time to read a draft of this paper and for providing their own insights and comments.

\end{document}